\documentclass[11pt]{amsart}
\usepackage{amsmath,amssymb,latexsym,hyperref,url,graphicx}

\graphicspath{{figures/}}

\newtheorem{corollary}{Corollary}
\newtheorem{definition}{Definition}

\newtheorem{example}{Example}
\newtheorem{lemma}{Lemma}

\newtheorem{remark}{Remark}
\newtheorem{theorem}{Theorem}

\newcommand{\prob}[1]{\mathbb{P}\left(#1\right)}

\def\Z{\mathbb{Z}}

\def\ds{\displaystyle}
\newcommand{\abs}[1]{\left|#1\right|}

\usepackage[usenames,dvipsnames]{xcolor} 

\newcommand{\se}[1]{\prec_{#1}}

\begin{document}

\title{Deterministic walk in an excited random environment}

\author{Ivan Matic}
\address{Department of Mathematics \\
	Baruch College \\
	New York, NY 10010, USA}

\author{David Sivakoff}
\address{
	Department of Statistics and Department of Mathematics \\
	Ohio State University \\
	Columbus, OH 43210, USA}

\date{\today}

\begin{abstract} Deterministic walk in an excited random environment is a non-Markov integer-valued process $(X_n)_{n=0}^{\infty}$, whose jump at time $n$ depends on the number of visits to the site $X_n$. 
The environment can be understood as stacks of cookies on each site of $\mathbb Z$. Once all cookies are consumed at a given site, every subsequent visit will result in a walk taking a step according to the direction prescribed by the last consumed cookie. If each site has exactly one cookie, then the walk ends in a loop if it ever visits the same site twice. If the number of cookies per site is increased to two, the walk can visit a site infinitely many times and still not end in a loop. Nevertheless the moments of $X_n$ are sub-linear in $n$ and we establish monotonicity results on the environment that imply large deviations.
\end{abstract}

\maketitle

\section{Introduction}
The deterministic walk in excited random environment in one dimension is a discrete time process, $(X_t)_{t=0}^{\infty}: \Omega \to \Z^{\{0, 1, \ldots\}}$.  For $L,M \in \mathbb N$, the set of environments is
\begin{align*}
\Omega = \Omega(L,M)=& \left\{\omega\in [-L,L]^{\Z_{\ge 0} \times \Z}\right.: \\
& \qquad \omega(j,z)=\omega(M-1,z) \text{ for each } j\geq M-1 \text{ and each } \left. z\in\Z \right\},
\end{align*}
where $[a, b] := \{a, a+1, \ldots, b\}$.  We imagine $\Omega$ as stacks of $M$ {\em cookies}, $\omega(0,z), \dots, \omega(M-1,z)$, at each site $z\in \Z$, each with an arrow pointing to the right or to the left by at most $L$.  We assume that $\Omega$ is equipped with the product measure $\mathbb{P} = \mathbb{P}_{L,M}$ such that $\{\omega(j, z) : j\in [0,M-1], z\in \Z\}$ are i.i.d.~with distribution $\mu$ supported on $[-L,L]$.  Note the abuse of notation here, that $\omega\in \Omega$ is both an element of the set of environments, and a random element (via the identity map) with distribution $\mathbb{P}$.  We further assume that $\mu(k)>0$ for all $k\in [-L,L]$.

To define the {\em deterministic walk in excited random environment}, first let $L_t(z) = L_t(\omega,z)$ denote the number of times that the walker visited $z$ in the time interval $[0,t-1]$, $$L_t(z)=\left|\left\{
0\leq j<t:X_j=z
\right\}\right|,$$
where $|A|$ denotes the cardinality of the set $A$. For each $\omega \in \Omega$, we define $X_t = X_t(\omega)$ recursively~as
$$
\begin{aligned}
X_0 &= 0,\\
X_{t+1} &= X_t + \omega(L_t(X_t), X_t).
\end{aligned}
$$

 
The main result of this paper is the large deviations estimate of the probability that $X_n$ is located at a distance of order $O(n)$ from the origin.

\begin{theorem}\label{large_deviations} Fix $M\geq 3$. There exists a function $\phi:[0,L]\to\mathbb (-\infty,0]$ such that for each~$\lambda\in \mathbb [0,L]$ 
\begin{eqnarray}
\label{main_limit}\lim_{n\to\infty}\frac1n\log\prob{X_n\geq \lambda n}=\phi(\lambda).
\end{eqnarray}
\end{theorem}
\begin{remark}
The assumption of an i.i.d.~environment can be weakened, and we make this assumption merely for the ease of exposition.  For instance, the proof of Theorem~\ref{large_deviations} holds with minor modification if the cookies at a given site are jointly distributed such that every combination of $M$ cookies has strictly positive probability, while the cookies at distinct sites are independent and identically distributed. In particular, this includes the case where each `layer' of cookies has a different distribution.  Also with minor modifications to the proofs, all of our results can be established under the weaker assumption that  $\mu(k)>0$ for $k\in [-L,L]\setminus\{0\}$. 
\end{remark}
\begin{remark}
The function $\phi$ is concave on $[0,L]$, with $\phi(0)=0$ and $\phi(\lambda)<0$ for $\lambda \in (0, L]$.  This is proved in section~\ref{concave}.
\end{remark}
\begin{remark}
We expect Theorem~\ref{large_deviations} to hold when $M=2$, and can prove that it does when $L\leq 2$.  However, our proof for $M\geq 3$ does not work when $M=2$ and $L\geq 3$.  The case $M=1$ was proved using a different method in~\cite{matic11}.
\end{remark}

The model studied in this paper traces its origins to the study of stochastic partial differential equations. The viscosity solutions to random Hamilton-Jacobi and Hamilton-Jacobi-Bellman equations can be represented using variational formulas \cite{scott_hung_yifeng_2014, krv, souganidis_1999}. The controls in the formulas are solutions to ordinary differential equations or stochastic differential equations in random environments whose discrete analogs are deterministic walks in random environments (DWRE) and random walks in random environments (RWRE), respectively \cite{rezakhanlou_rwre}. 

DWERE is a non-Markov process that generalizes DWRE in the same way as RWERE generalizes RWRE by allowing several cookies on each site. Large deviations for RWRE were studied in the past and various results were obtained \cite{petersonejp12,  rassoulaghaseppalainenyilmaz13, varadhan04, yilmaz11, yilmazzeitouni10}. The approaches from these papers cannot be applied to DWRE or DWERE because  the latter models do not possess the quenched ellipticity property.  
Results related to the laws of large numbers for non-elliptic random walks were established in  \cite{denhollanderdossantossidoravicius13}. 
In the case of RWRE, one can assume ellipticity and use the point of view of the particle to see the process as a Markov chain on a probability space with sufficient compactness to apply the Donsker--Varadhan theory \cite{varadhan_rwre, yilmaz11}.  Large deviations analogous to Theorem~\ref{large_deviations}, but in all dimensions, were proved for DWRE by an analysis of loops\cite{matic11}. However, this loop analysis is not applicable to large deviations of DWERE. 
The process that we are studying is also related to the Lorentz mirror model. In the case of the mirror problem it is conjectured that the paths are almost surely finite \cite{grimmett_menshikov_volkov_labyrinths} which in our model is an easy consequence of Lemma \ref{lemma_bound_annulus}. Recent progress on the mirror problem establishes a lower bound on the probability that the rays reach distances of order $n$  \cite{kozma_sidoravicius_mirror}.

RWERE was introduced in the paper of Benjamini and Wilson \cite{benjamini_wilson}. In more general versions of excited random walks, the number of cookies per site is greater than one. A number of results were established about recurrence, balisticity, monotonicity, and return times to zero \cite{elena_tom_11, elena_martin_08, elena_martin_13, elena_martin_14,  peterson_monotonicity}. Some of these excited random walks are known to converge to Brownian motion \cite{dima_elena_12}. Large deviations for random walks in excited random environments were studied in \cite{peterson12}. In the case of random walks in excited random environments very little is known in dimensions greater than 1 and in the cases when the steps are not nearest-neighbor. Our main proof is also restricted to dimension 1, however we are allowing our walk to make jumps of sizes bigger than 1. 

The case $M=1$ corresponds to DWRE and Theorem \ref{large_deviations} can be obtained in arbitrary dimension $d$ \cite{matic11}. The main argument of the proof used the fact that once the walk visits a site it has visited before, it will end in a loop. This can be simply stated as the {\em $0-1-\infty-$principle}, meaning that in DWRE the number of times a given site can be visited by the walk is zero, one, or infinity. However, we will see in Theorem \ref{weird_theorem} that DWERE is a much richer model, and that a site can be visited arbitrarily many times.

The key ingredient in the proof of Theorem \ref{large_deviations} is Lemma \ref{main_lemma_decreasing} that establishes a monotonicity property among {\em favorable environments}. A configuration of cookies on $\mathbb Z$ is called a favorable environment if it enables the walk starting at $0$ to reach $\lambda n$ in fewer than $n$ steps. Lemma \ref{main_lemma_decreasing} states that for every favorable environment one can change several cookies in $[0, O(\sqrt n)]$ to make another favorable environment that also allows the walk to avoid any backtrackings over $0$. This result was the key to establishing a sub-additivity necessary for proving large deviations. 

In the case when the maximal jump size is $L=2$ one can replace $O(\sqrt n)$ in Lemma \ref{main_lemma_decreasing} with a finite number. It remains unknown whether $O(\sqrt n)$ can be replaced by a finite number when $L\geq 3$.
%

Before delving into properties of the model, it is instructive to consider one concrete example. 

\begin{example}
Assume that the random environment is created in the following way. Each site of $\Z$ independently choses a sequence of two integers from $\{-3,-2, \dots, 3\}$. In the picture below the site $0$ has cookies $(-3, 2)$, while the site $2$ has cookies $(-2,1)$. We will denote the cookies at $0$ by $\omega(0,0)=-3$ and $\omega(1,0)=2$. Similarly, $\omega(0,2)=-2$ and $\omega(1,2)=1$.

\begin{center}
\includegraphics[scale=0.5]{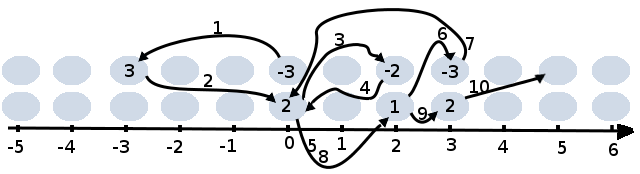} 
\end{center}
If the cookies are as shown in the picture above, then the first 10 steps of the walk are $X_0=0$, $X_1=-3$, $X_2=0$, $X_3=2$, $X_4=0$, $X_5=2$, $X_6=3$, $X_7=0$, $X_8=2$, $X_9=3$, and $X_{10}=5$.
\end{example}

\section{Properties of excited random environments}

The results in this section serve to outline some of the major differences between excited and non-excited environments. In regular deterministic walks in non-excited random environments, the number of visits to any particular site can be $0$, $1$, or infinity. The last case corresponds to the situation in which the walk ends in a loop passing through a prescribed number of sites infinitely many times.  In an excited environment, the walker may revisit $0$, for instance, any number of times $1, 2, \ldots, \infty$.  However, the probability of revisiting $0$ a large finite number of times decays exponentially, as the next theorem demonstrates. For convenience, we let \begin{eqnarray*}\mu_{\min} &=& \min\{\mu(k) : k\in [-L,L]\}, \text{ and}\\ \mu_{\max} &=& \max\{\mu(k) : k\in [-L,L]\}.\end{eqnarray*}

\begin{theorem} \label{weird_theorem}
Assume that $L\geq 2$ and $M\geq 2$.  Let $D_0$ be the cardinality of the set $\{n: X_n=0\}$. For each $k\in\mathbb N$ the following inequality holds
\begin{equation*}
(\mu_{\min})^{4Mk} \leq \prob{D_0 = k} \leq 2(1-(\mu_{\min})^{2M+L-2})^{k/2LM}
\end{equation*}
\end{theorem}
\begin{proof}

The lower bound follows from Lemma~\ref{weird_lemma} below.

\begin{lemma}\label{weird_lemma} There exist two functions $f,g:\mathbb Z\to\{-2,-1,1,2\}$ such that the deterministic sequence $x_n$ defined by $x_0=0$ and 
$$x_{n+1}=x_n+\left\{\begin{array}{ll}f(x_n),& \mbox{if }x_n\in\{x_0,\dots, x_{n-1}\},\\
g(x_n),& \mbox{if }x_n\not\in\{x_0,\dots, x_{n-1}\}\end{array}\right.$$ contains exactly $k$ terms equal to $0$ and has $-2k\leq x_n \leq 2k-1$ for all $n$.
\end{lemma}

Indeed, if we find two such functions, then the event $E\subset\{D_0 = k\}$ can be constructed as follows:
$$
E = \left\{\omega \in \Omega \ : \ \omega(0,z)=g(z) \text{ and } \omega(i,z)=f(z) \text{ for } i\geq 1 \text{ and } -2k \leq z \leq 2k-1\right\}.
$$
We have $\prob{E} \geq (\mu_{\min})^{4Mk} > 0$, so $\prob{D_0 = k} \geq (\mu_{\min})^{4Mk} $.

For the upper bound, suppose that $V_0^j$ is the time of the $j$th visit to $0$.  If $V_0^k < \infty$ and $V_0^{k+1} = \infty$, then the walker cannot get stuck in a loop that includes $0$, but clearly must return to $0$ $k-1$ times.  Therefore, between consecutive visits to $0$, the walker must see at least one new cookie, otherwise it will be stuck in a loop containing $0$.  That is, for each $0\leq j\leq k-1$, there exists $x \in \left\{X_{V_0^j}, X_{V_0^j+1}, \ldots, X_{V_0^{j+1}}\right\}$ such that $L_{V_0^j}(x) \leq M-1$.  Therefore, by time $V_0^k$, the walker must have visited at least $k / M$ distinct vertices.  Furthermore, this implies that the walker must have visited at least $k/LM$ regions of the form $[iL, (i+1)L -1]$ for $i \in \Z$.  That is,
$$
\abs{\left\{ i \in \Z: [iL,(i+1)L - 1] \cap \{X_t : 0\leq t \leq V_0^k\}\neq \emptyset\right\}} \geq \frac{k}{ LM}.
$$

In order for the walker to revisit 0 at time $V_0^k$, none of the regions $[iL, (i+1)L-1]$ that the walker visits before this time can be a trap where the walker gets stuck in a loop.  An example of a trapping configuration on the interval $[iL, (i+1)L-1]$ has $\omega(j,iL) = 1 = -\omega(j,iL+1)$ for $j\geq 0$ and $\omega(0,iL+x) = -x$ for $x = 2, \ldots, L-1$.  Therefore, the probability that $[iL, (i+1)L-1]$ is a trapping region is at least $(\mu_{\min})^{2M+L-2}$.

Finally, observe that the set of $i\in \Z$ such that the walker visits $[iL, (i+1)L-1]$ by time $V_0^k$ must be a set of consecutive integers containing $0$, since the walker cannot jump over any such region.  Therefore, the walker must either visit every such region for $0\leq i \leq k/2LM - 1$, or every such region for $-k/2LM+1 \leq i \leq 0$.  The probability that none of these regions is a trap gives the upper bound.
\end{proof}

\begin{proof}[Proof of Lemma~\ref{weird_lemma}] 
Let us first define $f$ and $g$ on the set $\mathbb Z_-$ of negative numbers, i.e. $\mathbb Z_-=\{\dots, -3,-2,-1\}$. If $z\in\mathbb Z_-$ is odd we set $f(z)=g(z)=2$, and if $z\in\mathbb Z_-$ is even we set $f(z)=-2$ and $g(z)=1$.

For $i\in\{0,1,\dots, 2k-3\}$ we define $$g(i)=\left\{\begin{array}{ll}-2,&\mbox{ if } i \mbox{ is even,}\\
-1,&\mbox{ if } i \mbox{ is odd;}\end{array}\right.\quad\quad\mbox{and}\quad\quad f(i)=\left\{\begin{array}{rl}-2,&\mbox{ if } i \mbox{ is even,}\\
2,&\mbox{ if } i \mbox{ is odd.}\end{array}\right.$$
We finally define $f(2k-1)=g(2k-1)=-1$ and $f(2k-2)=g(2k-2)=1$. 
\begin{center}
\includegraphics[scale=0.5]{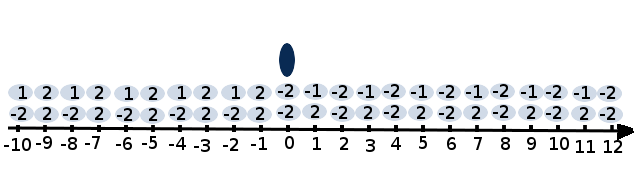} 
\end{center}

We will prove that $x_{2i(i+1)}=0$ for $i\in\{0,1,2,\dots,k-1\}$ and that all other terms of the sequence $(x_n)_{n=0}^{\infty}$ are non-zero. 

We will now use induction on $i$ to prove that for each $i\in\{0,1,\dots, k-1\}$ the following holds: \begin{eqnarray}\nonumber \label{induction_proof}&&x_{2i(i+1)}=0\;\mbox{ and }\\ &&\{x_0,\dots, x_{2i(i+1)}\}=\{-2i,-2i+1,\dots, 0,1,\dots, 2i-1\}.\end{eqnarray}

This is easy to verify for $i=0$ and $i=1$. Assume that the statement is true for some $i$ and let us prove it for $i+1$. 

Let us denote $R_i=\{x_0,\dots, x_{2i(i+1)}\}=\{-2i,\dots, 2i-1\}$.
Then we have that $x_{2i(i+1)}=0$, and since $0\in R_i$ we have that $x_{2i(i+1)+1}=0+f(0)=-2$. Since $-2\in R_i$ we get $x_{2i(i+1)+2}=-2-2=-4$, and so on. 
We obtain that $x_{2i(i+1)+i}=-2i\in R_i$ which implies that $x_{2i(i+1)+i+1}=-2i-2\not\in R_i$. Therefore $x_{2i(i+1)+i+2}=-2i-2+g(-2i-2)=-2i-2+1=-2i-1\not\in R_i$. 
Hence $x_{2i(i+1)+i+3}=-2i-1+g(-2i-1)=-2i+1\in R_i$. This implies that $x_{2i(i+1)+i+4}=-2i+3\in R_i$. Continuing this way we obtain that $x_{2i(i+1)+i+2i+2}=2i-1\in R_i$ and
$x_{2i(i+1)+i+2i+3}=2i+1\not\in R_i$. Therefore $x_{2i(i+1)+3i+4}=2i+1+g(2i+1)=2i\not\in R_i$ and $x_{2i(i+1)+3i+5}=2i+f(2i)=2i-2\in R_i$. 

We now have $x_{2i(i+1)+3i+6}=2i-4\in R_i$ and continuing this way we obtain $x_{2i(i+1) +3i+i+4}=0$. This implies that $x_{2i(i+1)+4(i+1)}=0$ which is the same as $x_{2(i+1)(i+2)}=0$.
In addition, $$\{x_0,\dots, x_{2(i+1)(i+2)}\}=R_i\cup\{-2(i+1),-2i-1,2i,2i+1\}=\{-2i-2,\dots, 2i,2i+1\}$$ thus the proof of (\ref{induction_proof}) is complete. 

Placing $i=k-1$ in the first equation in (\ref{induction_proof}) we obtain $x_{2k(k-1)}=0$, and similarly as in the previous proof we get that $x_{2i(i+1)+3i+4}=2i=2k-2$. However, since $g(2k-2)=1$ we get that $x_{2i(i+1)+3i+5}=2k-1$ and subsequently that $x_{2i(i+1)+3i+6}=2k-1+g(2k-1)=2k-2$. This implies that $x_{2i(i+1)+3i+7}=2k-2+f(2k-2)=2k-1$ and $x_{2i(i+1)+3i+8}=2k-1+f(2k-1)=2k-2$. From now on the sequence is periodic and none of the terms will be zero. 

This proves that there are exactly $k$ terms equal to 0, and since it is stuck in a loop, no vertices outside $[-2k, 2k-1]$ are visited.
\end{proof}

\section{Laws of large numbers} 

In this section we assume that the walk is in $\mathbb R^d$ for any $d\in\mathbb N$. 
We prove that the walk is almost surely bounded. As a consequence, the law of large numbers holds with the limiting velocity equal to $0$. Moreover, all of the moments of the process $X_n$ have growth that is slower than any function $f(n)$ that satisfies $\lim_{n\to\infty}f(n)=+\infty$. This means that the central limit theorem also does not have the classical form for this model. 

The following lemma will be essential for the proofs of the boundedness of the walk. This lemma establishes the exponential decay of the probabilities that the walk reaches the {\em annulus} $A_k$ defined in the following way:
$$A_k=\left[-(k+1)L,(k+1)L\right]^d\setminus \left[-kL,kL\right]^d.$$

This way, $A_0$ is the cube $[-L,L]^d$, while for $k\geq 1$, $A_k$ is an annulus. 

\begin{lemma}\label{lemma_bound_annulus} There exists a positive real number $c\in(0,1)$ and an integer $k_0$ such that $$\prob{T_{A_k}<+\infty}\leq c^k$$ holds for all $k\geq k_0$.
\end{lemma}
\begin{proof}
For $x\in\mathbb Z^d$ let us denote by $G(x)$ the event that all cookies at $x$ are zero-cookies. In other words, 
$G(x)=\left\{\omega(x,i)=0 \mbox{ for all } 0\leq i\leq M-1\right\}$. On the event $G(x)$ the walk would get stuck at the location $x$ if it ever reaches it. 

We obviously have the following relation $$\prob{T_{A_{k+1}}<+\infty}\leq \prob{T_{A_{k}}<+\infty, G\left(X_{T_{A_{k}}}\right)^C }.$$
The required inequality would be established if we manage to prove that for every $k\geq 0$ the following inequality holds:
\begin{eqnarray}\label{recursive_equality}
 \prob{T_{A_{k}}<+\infty, G\left(X_{T_{A_{k}}}\right)^C }&\leq&\left(1-\mu_{\min}^M\right)\cdot \prob{T_{A_{k}}<+\infty  }.
\end{eqnarray}
For each $x\in A_k$ let us introduce the following event $$\Omega_x=\left\{T_{A_k}<+\infty\mbox{ and } X_{T_{A_k}}\left(\omega\right)=x\right\}.$$ The event $\Omega_x$ is in the sigma field generated by the cookies inside the set $A_0\cup\cdots \cup A_{k-1}$. Therefore, $\Omega_x$ and $G\left(x\right)$ are independent.

We now have
\begin{eqnarray*}
 \prob{T_{A_{k}}<+\infty, G\left(X_{T_{A_{k}}}\right)^C }&=&\sum_{x\in A_{k}}\prob{\Omega_x\cap G\left(x\right)^C}\\
 &=&\sum_{x\in A_k}\prob{\Omega_x}\cdot \prob{G\left(x\right)^C}\\
 &\leq&\left(1-\mu_{\min}^M\right) \cdot\sum_{x\in A_k}\prob{\Omega_x}\\
 &=&\left(1-\mu_{\min}^M\right)\cdot \prob{T_{A_{k}}<+\infty}.
\end{eqnarray*}
This completes the proof of (\ref{recursive_equality}), and hence the proof of the required inequality.
\end{proof}

A consequence of Lemma \ref{lemma_bound_annulus} is that the sequence $X_n$ is almost surely bounded. We present this result in the following lemma.

\begin{lemma}\label{loop_probability_1} Denote by $B$ the event that $X_n$ is a bounded sequence. More precisely, $B=\{\exists M$ such that $\|X_n\|_{\infty}\leq M$ for all $n\geq 0\}$, where $\|x\|_{\infty}$ denotes the biggest coordinate of the $d$-dimensional vector $x$. 
Then $\prob B=1$.
\end{lemma}
\begin{proof}  
On the event $B^C$ we must have   $\left\{T_{A_k}<+\infty\right\}$ for all $k\in\mathbb N$. However, Lemma \ref{lemma_bound_annulus} implies that $\prob {T_{A_k}<+\infty}<c^k$ for each $k\geq k_0$, hence
$$\prob{B}=\prob{\bigcap_{k\geq 1} \left\{T_{A_k}<+\infty\right\}}\leq c^k,$$ for every $k\geq k_0$ which is only possible if $\prob{B}=0$.
\end{proof}

\begin{corollary}\label{theorem_bound_moments} For every function $f:\mathbb N\to\mathbb R$ such that $\lim_{n\to\infty}f(n)=+\infty$ the following limit holds almost surely: \begin{eqnarray}
\nonumber\lim_{n\to\infty}\frac{  \left\|X_n\right\|_{\infty} }{f(n)}=0.\end{eqnarray}
\end{corollary}

\section{Large deviations }
In this section we prove Theorem \ref{large_deviations}. For $\lambda\in[0,L]$ we want to show the existence of the limit $\ds \lim_{n\to\infty} \frac1n \log \prob{X_n\geq \lambda n}$.  
As stated earlier, we will prove this under the assumption that there are at least $3$ cookies on each site, i.e. $M\geq 3$.
Before we can prove the theorem we need to introduce the following notation. For $k\in\mathbb N$ and $x\in\mathbb Z$ let us denote by $V_{x}^k$ the time of the $k$th visit to the site $x$.  The hitting time $V_x^k$ can be inductively defined as: 
\begin{eqnarray*}
V_x^1(\omega)&=&\inf\{m: X_m(\omega)=x\},\\
V_x^{i+1}(\omega)&=&\inf\{m> V_x^i(\omega): X_m(\omega)=x\} \;\mbox{for }i\geq 1.
\end{eqnarray*}
Instead of $V_x^1$ we will often write $V_x$. For any $A\subseteq \mathbb R$ let us define $$T_A=\inf\{n: X_n\in A\}.$$ 
If $x>0$ we will write $T_x$ instead of $T_{[x,+\infty)}$. The following two  inequalities are easy to establish: 
\begin{align}
\label{inequality 1}
\prob{X_n \geq \lambda n}& \leq \prob{T_{\lambda n} \leq n} \;\;\mbox{and}\\
\label{inequality 2}
\prob{X_n\geq \lambda n}& \geq \prob{T_{\lambda n}\le n, \omega(j,x)=0 \mbox{ for all }j \mbox{ and } x\in [{\lambda n}, {\lambda n} + L]} \\
\nonumber &\geq C \prob{T_{\lambda n} \leq n},
\end{align}
for some constant $C$ independent of $n$. 
Therefore, it is sufficient to prove that $\ds \lim_{n\to\infty} \frac1n \log \prob{T_{\lambda n} \leq n}$ exists.

Let $$A_n := \left\{T_{\lambda n} \leq n, \inf_{k\leq T_{\lambda n}}X_k\geq 0 \right\}$$ denote the event that the walk reaches $\lambda n$ by time $n$ before backtracking to the left of $0$.  It is trivially true that $A_n \subset \{T_{\lambda n}\leq n\}$ so $\prob{A_n} \leq \prob{T_{\lambda n} \leq n}$.

\subsection{Definitions} If $a = (a_\ell)_{\ell = 1}^K\in \mathbb{Z}$ where $K \in \mathbb{N}\cup\{\infty\}$ and $B\subset\mathbb{Z}$, then the restriction of $a$ to $B$ is denoted $a\big|_{B}$, and is the sequence of terms in $a$ that belong to $B$ with their order intact.  For $t_1\leq t_2$, let $$X_{[t_1,t_2]}(\omega) = (X_{t_1}(\omega), X_{t_1+1}(\omega), \ldots, X_{t_2}(\omega))$$ denote the sequence of locations of the walker from steps $t_1$ through $t_2$.

\begin{definition}\label{subenvironment}
For $\omega, \omega' \in \{T_\ell < \infty\}$ and $0\leq m < \ell$, let $\omega' \se{\ell,m} \omega$ denote the following relationship between environments $\omega$ and $\omega'$.
\begin{enumerate}
\item $\omega'(j,x) = \omega(j,x)$ for all $x>m$ and all $j\ge0$; 
\item $X_{[0,T_{\ell}(\omega')]}(\omega') \big\vert_{[m,\ell]} = X_{[0,T_{\ell}(\omega)]}(\omega) \big\vert_{[m,\ell]}$;
\item The sequence $X_{[0,T_{\ell}(\omega')]}(\omega')$ is a subsequence of $X_{[0,T_{\ell}(\omega)]}(\omega)$.
\end{enumerate}
\end{definition}

In other words, we will write $\omega' \se{\ell,m} \omega$ if (1) the two environments are identical to the right of $m$, (2) the walkers on both environments visit the same sites in the same order to the right of $m$ and until exceeding $\ell$, but (3) the walker on $\omega'$ may avoid some parts of the path followed by the walker on $\omega$ to the left of $m$. Observe that $\se{\ell,m}$ gives a partial ordering of the environments in $\{T_\ell < \infty\}$. 

\subsection{Monotonicity results} The next theorem provides the asymptotic equivalence of probabilities $\prob{T_{\lambda n}\leq n}$ and $\prob{A_n}$ on the logarithmic scale. 
\begin{theorem}\label{main_bound}
There exists $C\in\mathbb R_+$, depending on $L, M$ and $\mu$, such that the following inequality holds for all~$n$:
\begin{eqnarray}\label{main inequality} C^{\sqrt n} \prob{A_n}\geq \prob{T_{\lambda n}\leq n}.\end{eqnarray}
\end{theorem}
\begin{proof}
We will use the following result whose proof will be presented later.
\begin{lemma}\label{main_lemma_decreasing} Assume that $n>\left(\frac{2L}{\lambda}\right)^2$. For each $\omega\in \{T_{\lambda n}\leq n \}$  there exists $\omega'\in\{T_{\lambda n }\leq n\}$ such that $$\omega'\se{\lambda n, 2L\sqrt n}\omega \quad \mbox{and}\quad X_{[0,T_{\lambda n}(\omega')]}(\omega')\cap (-\infty, -1]=\emptyset.$$
\end{lemma}

 For given $\omega \in \{T_{\lambda n}\leq n\}$   we can apply Lemma \ref{main_lemma_decreasing}  to obtain a new environment $\hat \omega\in \{T_{\lambda n}\leq n\}$ such that 
$$\inf_{0\leq k\leq T_{\lambda n}(\hat\omega)}X_k(\hat \omega)= 0.$$

Let us denote by $\tilde \omega$ the environment defined by: 
\begin{enumerate}
\item[(i)] For $x\not \in [0,2L\sqrt n]$ and  $j\in\{0,\dots, M-1\}$: $\tilde \omega(j,x)=\omega(j,x)$.
\item[(ii)] For $x\in [0,2L\sqrt n]$ and $j\in\{0,\dots, M-1\}$: $\tilde \omega(j,x)=\hat\omega(j,x)$.
\end{enumerate}
Since $\hat \omega$ and $\tilde \omega$ coincide on sites in $[0,+\infty)$ and $X(\hat\omega)$ does not visit negative sites, we conclude that $X(\tilde\omega)$ does  not visit negative sites. Therefore, for each 
$\omega\in\{ T_{\lambda n}\leq n\}$ there exists $\tilde \omega\in A_n$ such that $\omega$ and $\tilde \omega$ coincide on all sites except possibly for the sites in $[0,2L\sqrt n]$.

We can now  define a function $f:\{T_{\lambda n}\leq n\}\to A_n$ in the following way. 

Let us fix $n$. We can now define $\hat{\mathbb P}$ on the restriction $\hat\Omega$ of $\Omega$ that corresponds to the portion of the integer axis between the numbers $-Ln$ and $Ln$.  
The purpose of this restriction is so that $\hat{\mathbb P}(\omega)>0$ for each $\omega \in \hat\Omega$. Formally, $$\hat \Omega=[-L,L]^{[0,M-1]\times [-Ln, Ln]},$$ and $\hat{\mathbb P}$ is defined to be the restriction of $\mathbb P$. Then we have $\hat{\mathbb P}\left(T_{\lambda n}\leq n\right)=\prob{T_{\lambda n}\leq n}$ and $\hat{\mathbb P}(A_n)=\prob{A_n}$, where each $\omega \in \Omega$ is identified with an element of $\hat\Omega$ by truncation, which will also be denoted $\omega$. 
It suffices to prove that there is $C\in\mathbb R_+$ (independent of $n$) such that \begin{eqnarray}\label{inequality with hats}\hat{\mathbb P}\left(T_{\lambda n}\leq n\right)\leq C^{\sqrt n}\hat{\mathbb P}(A_n) .\end{eqnarray} Let $C_1=\left(\frac{\mu_{\max}}{\mu_{\min}}\right)^M$ and   $C_2=(2L+1)^M$.
We will prove inequality (\ref{inequality with hats}) for $C=\left(C_1C_2\right)^{2L}$. 
\begin{eqnarray*}\hat{\mathbb P} \left(T_{\lambda n}\leq n\right)&=& \sum_{\omega \in\left\{T_{\lambda n}\leq n\right\} }\hat{\mathbb P}(\omega)\leq  \sum_{\omega \in\left\{T_{\lambda n}\leq n\right\} }C_1^{2L\sqrt n}\hat{\mathbb P}(f(\omega))\\
&=&C_1^{2L\sqrt n}\sum_{\omega \in\left\{T_{\lambda n}\leq n\right\} }\sum_{\omega'\in A_n}\hat{\mathbb P}(\omega')\cdot 1_{f(\omega)=\omega'}\\&=&
C_1^{2L\sqrt n}\sum_{\omega'\in A_n}\sum_{\omega \in\left\{T_{\lambda n}\leq n\right\} }\hat{\mathbb P}(\omega')\cdot 1_{f(\omega)=\omega'}\\&=&
C_1^{2L\sqrt n}\sum_{\omega'\in A_n}\hat{\mathbb P}(\omega')\cdot\sum_{\omega \in\left\{T_{\lambda n}\leq n\right\} } 1_{f(\omega)=\omega'}\\&=&
C_1^{2L\sqrt n}\sum_{\omega'\in A_n}\hat{\mathbb P}(\omega')\cdot\left|\left\{f^{-1}(\omega')\right\}\right|.
\end{eqnarray*} 
If $f(\omega)=\omega'$ then the environments $\omega$ and $\omega'$ coincide outside of $[0, 2L\sqrt n]$. Since there could be at most $C_2^{2L\sqrt n}$ different environments that coincide with $\omega$ outside of $[0, 2L\sqrt n]$, we obtain $$\hat{ \mathbb P}\left(T_{\lambda n}\leq n\right)\leq 
C_1^{2L\sqrt n}\cdot C_2^{2L\sqrt n} \sum_{\omega'\in A_n} \hat{\mathbb P}(\omega')=C^{\sqrt n}\hat{\mathbb P}(A_n).$$
This completes the proof of inequality (\ref{inequality with hats}) which is equivalent to (\ref{main inequality}).
\end{proof}

\begin{lemma}\label{raising stacks}
Fix $\omega \in \Omega$.  Suppose $a,b \in \mathbb{Z}$ with $\abs{a-b}\leq L$, and $0\leq t_a<t_b$ are such that $X_{t_a}(\omega) = a$, $X_{t_b}(\omega) = b$, and one of the following two conditions is satisfied:
\begin{enumerate}
\item[(a)] $L_{t_a}(\omega,a) < M-1$;
\item[(b)] $X_{t_a+1}(\omega)=b$. 
\end{enumerate} Then there exists $\omega' \in \Omega$ such that
\begin{itemize}
\item[(i)] $X_{[0,t_a]}(\omega') = X_{[0,t_a]}(\omega)$;
\item[(ii)] $X_{[t_a+1, \infty)}(\omega') = X_{[t_b,\infty)}(\omega)$;
\item[(iii)] $\omega'(j,x) = \omega(j,x)$ for all $x \notin X_{[t_a,t_b]}(\omega)$ and all $j\geq 0$.
\end{itemize}
\end{lemma}

\begin{proof}
Let $C_x = L_{t_b}(\omega,x) - L_{t_a}(\omega,x)$ be the number of times the site $x$ is visited by the sequence $X_{[t_a,t_b]}(\omega)$.  Under the assumption (a) we obtain the environment $\omega'$ from $\omega$ by removing the cookies visited by the walker $X(\omega)$ in the time interval $[t_a+1,t_b]$ and rewiring the top cookie at $a$ at time $t_a$ to point at $b$.  That is,
$$
\omega'(j,x) = \begin{cases}
\omega(j,x) & \hbox{for $x\notin X_{[t_a,t_b]}(\omega)$ and all $j\geq0$}\\
\omega(j,x) & \hbox{for $x\in X_{[t_a,t_b]}(\omega)$, $0\leq j < L_{t_a}(\omega,x)$} \\
\omega(j+C_x,x) & \hbox{for $x\in X_{[t_a,t_b]}(\omega)$, $x\neq a$, $j\geq L_{t_a}(\omega,x)$ } \\
\omega(j+C_a,a) & \hbox{for $x=a$, $j > L_{t_a}$} \\
b-a & \hbox{for $x = a$, $j = L_{t_a}(\omega,a)$}.
\end{cases}
$$
From the definition of $\omega'$, it is clear that (iii) is satisfied, and (i) is satisfied because, from the perspective of the walker, $\omega'$ and $\omega$ are identical up until time $t_a$.  Finally, (ii) is satisfied because the remaining environments at time $t_b$ in $X(\omega)$ and at time $t_a+1$ in $X(\omega')$ are identical.  The assumption (a)  guarantees that rewiring the top cookie at $a$ is allowed, since rewiring the $M^{\text{th}}$ cookie would require modifying all cookies $j\geq M-1$, which might affect the future path of the walk. If (b) is assumed instead of (a), then no rewiring is necessary since the cookie at $a$ at the time $t_a$ points to $b$ in both $\omega$ and $\omega'$, i.e. $\omega\left(L_{t_a}+C_a,a\right)=b-a$ and we can keep the same definition for $\omega'$ as when working under the assumption (a).
\end{proof}

\begin{proof}[Proof of Lemma \ref{main_lemma_decreasing}] Let $$\mathcal G(\omega)=\left\{\omega'\in \{T_{\lambda n}\leq n\}:\omega'\se{\lambda n, 2L\sqrt n}\omega\right\}.$$ Let $\sigma$ be the element (or one of the elements) of $\mathcal G(\omega)$ such that
 \begin{eqnarray}
 T_{\lambda n}(\sigma)=\min_{\omega'\in \mathcal G(\omega)}T_{\lambda n}(\omega').
 \label{minimality assumption}
\end{eqnarray}
The above $\sigma$ exists because $T_{\lambda n}(\omega')$ is an integer for each $\omega'$. 
If $X_{[0,T_{\lambda n}(\sigma)]}(\sigma)\cap (-\infty,-1]=\emptyset$ we may take $\omega'=\sigma$. Assume therefore that there exists $x<0$ such that $x\in X_{[0,T_{\lambda n}(\sigma)]}(\sigma)$.
Let us define \begin{eqnarray*}\alpha_1&=&\max\left\{0\leq k< V_{x}(\sigma): X_k(\sigma)>x\right\},\\
\beta_1&=&\min\left\{k>V_x(\sigma):X_k(\sigma)>x\right\}.
\end{eqnarray*}
The times $\alpha_1$ and $\beta_1$ are well defined because $x<0$ and $\sigma\in \{T_{\lambda n}\leq n\}$. 
Let $a_1=X_{\alpha_1}(\sigma)$ and $b_1=X_{\beta_1}(\sigma)$. Clearly, $a_1,b_1\in (x, x+L]$. 
Assume that $L_{\alpha_1}(\sigma,a_1)<M-1$. Then we can apply Lemma \ref{raising stacks} to $a=a_1$, $b=b_1$, $t_a=\alpha_1$, and $t_b=\beta_1$. The application of the lemma allows us to  obtain an environment $\sigma'\in \{T_{\lambda n}\leq n\}$ from the original environment $\sigma$ such that $\sigma'\se{\lambda n, 2L\sqrt n}\sigma$. Since one visit to $x$ is avoided in $\sigma'$ we would have $T_{\lambda n}(\sigma')\leq T_{\lambda n}(\sigma)-1$ which contradicts (\ref{minimality assumption}). Therefore we must have $L_{\alpha_1}(\sigma, a_1)\geq M-1$ (recall that this means the walker is visiting $a_1$ for at least the $M^{\text{th}}$ time). 
We now consider the sequence of times $V^1_{a_1}(\sigma)$, $V^2_{a_1}(\sigma)$, $\dots$, $V^{M}_{a_1}(\sigma)$  
at which the visits to $a_1$ have occurred. Assume that for some $s\in\{1,\dots, M-1\}$ we have $$X_{\left[V^s_{a_1}(\sigma), V^{s+1}_{a_1}(\sigma)\right]}(\sigma)\subseteq \left(-\infty, 2L\sqrt n\right].$$ By applying Lemma \ref{raising stacks} (with assumption (b)) to $a=X_{V^s_{a_1}(\sigma)-1}$, $b=a_1$, $t_a=V^s_{a_1}(\sigma)-1$, and $t_b=V^{s+1}_{a_1}$ we obtain an environment $\sigma'\in \{T_{\lambda n}<n\}$  such that 
$\sigma'\se{\lambda n, 2L\sqrt n}\sigma $ and $L_{T_{\lambda n}(\sigma')}(\sigma',a_1)\leq L_{T_{\lambda n}(\sigma)}(\sigma,a_1)-1$, which contradicts \eqref{minimality assumption}. 

Let us now define 
\begin{eqnarray*}
\alpha_2&=&\max\left\{0\leq k<V^2_{a_1}(\sigma): X_k(\sigma)>a_1\right\},\\
\beta_2&=&\min\left\{k>V^2_{a_1}(\sigma): X_k(\sigma)>a_1\right\}.
\end{eqnarray*}
We must have $a_1\leq L\sqrt n$. We can be certain that $\alpha_2$ is well defined because $X_{\left[V^1_{a_1}(\sigma),V^2_{a_1}(\sigma)\right]}\cap \left(2L\sqrt n,+\infty\right)\neq \emptyset$. The time $\beta_2$ is also well defined because $\sigma \in \{T_{\lambda n}\leq n\}$. Let $a_2=X_{\alpha_2}(\sigma)$ and $b_2=X_{\beta_2}(\sigma) $. According to the construction we must have $a_2, b_2\in (a_1,a_1+L]$.

Using the same argument as above we have that $L_{\alpha_2}(\sigma, a_2)\geq M-1$, and that for each $s\in\{1,2,\dots M-1\}$ we have $$X_{\left[V^s_{a_2}(\sigma),V^{s+1}_{a_2}(\sigma)\right]}(\sigma)\cap\left(2L\sqrt n, +\infty\right)\neq \emptyset.$$ Having defined $a_1 < \cdots < a_i$ and assuming that $a_i<L\sqrt n$ we inductively define the times $\alpha_{i+1}$ and $\beta_{i+1}$ in the following way: 
\begin{eqnarray*}
\alpha_{i+1}&=&\max\left\{0\leq k<V^2_{a_i}(\sigma): X_k(\sigma)>a_i\right\},\\
\beta_{i+1}&=&\min\left\{k>V^2_{a_i}(\sigma): X_k(\sigma)>a_i\right\}.
\end{eqnarray*}
Then we define $a_{i+1}=X_{\alpha_{i+1}}(\sigma)$ and $b_{i+1}=X_{\beta_{i+1}}(\sigma)$. Clearly, $a_{i+1},b_{i+1}\in (a_i,a_i+L]$. As above, we are certain that  
\begin{eqnarray}\label{all cookies eaten}L_{\alpha_{i+1}}(\sigma, a_{i+1})\geq M-1,\end{eqnarray} and for each $s\in\{1,2,\dots, M-1\}$ the following property holds: 
\begin{eqnarray}\label{far reaching excursion}X_{\left[V^s_{a_{i+1}}(\sigma),V^{s+1}_{a_{i+1}}(\sigma)\right]}(\sigma)\cap \left(2L\sqrt n, +\infty\right)\neq\emptyset.\end{eqnarray}


\begin{center}
\includegraphics[scale=0.5]{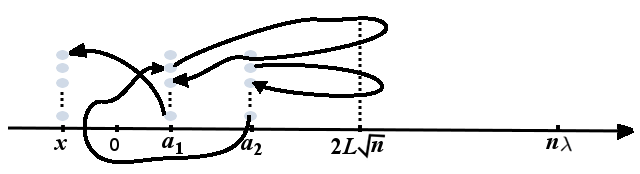} 
\end{center}

We can continue the induction until we have $x<a_1<\cdots<a_I$ where $I$ is the smallest index such that $a_I \geq L\sqrt{n}$.  Since $a_{i+1} - a_i \leq L$ for each $i \leq I-1$, we must have $I \geq \sqrt{n}$.  From \eqref{all cookies eaten}, we have that for each $i\le I-1$, before the second visit to the site $a_i$, the walk $X(\sigma)$ visits the site $a_{i+1}$ at least $M$ times. 
Furthermore, \eqref{far reaching excursion} implies that $V^{s+1}_{a_i}(\sigma)-V^s_{a_i}(\sigma)\geq \sqrt n$ for each $i\leq I$ and $s\leq M-1$, that is, the walk spends at least $\sqrt{n}$ steps between consecutive visits to each site $a_i$.


Now we will use our assumption that $M\geq 3$. We know that $a_1$ is visited at least three times before $x$ is visited for the first time. Between the second and third visit to $a_1$ the walk spent at least $\sqrt{n}$ steps. Therefore $V_x(\sigma)\geq V^3_{a_1}(\sigma)\geq V^2_{a_1}(\sigma)+ \sqrt n$. 
The second visit to $a_{1}$ has occurred after the  site $a_2$ is visited at least $M$ times, 
hence the second visit to $a_1$ occurred after the third visit to $a_2$. Therefore 
$V^2_{a_1}(\sigma)\geq V^3_{a_2}(\sigma)\geq V^2_{a_2}(\sigma)+\sqrt n$. Thus $V_x(\sigma)\geq V^2_{a_2}(\sigma)+2\sqrt n$. Since the second visit to $a_2$ occurred after $M$ visits to $a_3$ we know that the second visit to $a_2$ occurred after the third visit to $a_3$. Thus $V^2_{a_2}(\sigma)\geq V^3_{a_3}(\sigma)\geq V^2_{a_3}(\sigma)+\sqrt n$ and $V_x(\sigma)\geq V^2_{a_3}(\sigma)+3\sqrt n$. Continuing in this fashion, we obtain that $V_x(\sigma)\geq V^2_{a_I}(\sigma)+I\sqrt n\geq V^2_{a_I}(\sigma)+n\geq n$, which contradicts the assumption that $V_x(\sigma)<T_{\lambda n}(\sigma)\leq n$. This completes the proof of Lemma~\ref{main_lemma_decreasing}.
\end{proof}

\subsection{Large deviations} In this subsection we provide the proof to Theorem \ref{large_deviations}.
\begin{proof}[Proof of Theorem \ref{large_deviations}] It suffices to prove that $\prob{ A_{m+n}}\geq\prob{ A_n}\cdot \prob{A_m}$. 
We notice the following inclusion:
\begin{eqnarray*}
A_{n+m}&=&\left\{T_{\lambda (n+m)}\leq n, \inf \{X_k:0\leq k\leq T_{\lambda (n+m)}\}\geq 0\right\}\\
&\supseteq &A_n\cap \left\{T_{\lambda (n+m)}\leq n, \inf \{X_k:0\leq k\leq T_{\lambda (n+m)}\}\geq 0\right\}.\end{eqnarray*}
Let us define the walk $\hat X(\omega)$ for $\omega\in A_n\cap\left\{T_{\lambda n}\leq n, \inf \{X_k:0\leq k\leq T_{\lambda n}\}\geq 0\right\}$ in the following way: 
$\hat X_k(\omega)=X_{k+T_{\lambda n}}(\omega)-X_{T_{\lambda n}}(\omega)$. The walk $\hat X$ starts at $X_{T_{\lambda n}}$. In analogy to the stopping time $T_x$ for the walk $X$ we define  $\hat T_x$ for the walk $\hat X$. The precise definition is: 
$$\hat T_x(\omega)=T_{X_{T_{\lambda n}}+x}(\omega)-T_{\lambda n}(\omega).$$
In analogy to $A_n$ we define the event $\hat A_m$ for the walk $\hat X$:
$$\hat A_m=\left\{ \hat T_{\lambda m }\leq m, \inf\left\{\hat X_k:0\leq k\leq \hat T_{\lambda m}\right\}\geq 0
\right\}.$$
 On the event $A_n\cap \hat A_m$, by time $T_{\lambda n}+\hat T_{\lambda m}$ the walk $X$ reaches the site 
$X_{T_{\lambda n}}+\hat X_{\hat T_{\lambda m}}\geq \lambda \left(n+m\right)$. Therefore $A_n\cap \hat A_m\subseteq A_{n+m}$.  
We will now prove that $\prob {A_n\cap \hat A_m}=\prob{A_n}\cdot \prob{\hat A_m}$. For each $x\in[\lambda n, \lambda n+L]$, conditioned on $T_{\lambda n}=x$, the events $A_n$ and $\hat A_m$ are independent. Therefore
\begin{eqnarray*}
\prob{A_n\cap \hat A_m}&=&\sum_{x\in [\lambda n, \lambda n+L]}
\mathbb P\left(\left.A_n\cap \hat A_m\right|T_{\lambda n}=x\right)\cdot \prob{T_{\lambda n}=x}\\
&=&\sum_{x\in [\lambda n, \lambda n+L]}
\mathbb P\left(\left.A_n \right|T_{\lambda n}=x\right)\cdot\mathbb P\left(\left.  \hat A_m\right|T_{\lambda n}=x\right)\cdot \prob{T_{\lambda n}=x}.
\end{eqnarray*}
Since $\mathbb P\left(\hat A_m\left|T_{\lambda n}=x\right.\right)=\prob{\hat A_m}$ we obtain
\begin{eqnarray*}\prob{A_n\cap \hat A_m}&=&\mathbb P\left(  \hat A_m \right)\cdot \sum_{x\in [\lambda n, \lambda n+L]}
\mathbb P\left(\left.A_n \right|T_{\lambda n}=x\right)\cdot \prob{T_{\lambda n}=x}\\
&=&\prob{\hat A_m}\cdot \prob{A_n},
\end{eqnarray*}
which implies the  inequality
$$\prob{A_{n+m}}\geq \prob{A_n}\cdot \prob{A_m}$$ for all $n, m>0$.
The proof is completed using the inequalities (\ref{inequality 1}) and (\ref{inequality 2}) and Theorem \ref{main_bound}.
\end{proof}

\section{Case $L=2$ or $M=1$}
In the case when $L=2$ or the number of cookies per site is $0$, then we can obtain the exponential decay of probabilities $\prob{X_n\geq \lambda \xi(n)}$ for every positive function $\xi$ that satisfies $\xi(n)+\xi(m)\geq \xi(n+m)$. In particular this holds for $\xi(x)=x^{\theta}$ for $\theta\in(0,1)$. 
\begin{theorem} Let $\xi:\mathbb R_+\to\mathbb R_+$ be a positive super-additive function and assume that either $L= 2$ or $M=1$. Then there is a function $\varphi:\mathbb R_+\to \mathbb R$ such that for every $\lambda >0$ the following holds:
\begin{eqnarray}\label{limit_for_clt} \lim_{n\to\infty}\frac1n\log\prob{X_n\geq \lambda \xi(n)}=\varphi(\lambda).\end{eqnarray}
\end{theorem}
\begin{proof}
We will prove the theorem for the case $L=2$. The proof when $M=1$ is a simple generalization of the proof from the case of deterministic walks in random environments. First of all, the following inequality is obtained in the same way as in the proof of Theorem \ref{large_deviations}:
\begin{eqnarray}\label{ineq_incl1}
\liminf \frac1n\log\prob{T_{\lambda \xi(n)}\leq n}&\leq& \liminf \frac1n\log\prob{X_n\geq \lambda \xi(n)}\\ \limsup\frac1n\log\prob{X_n\geq \lambda \xi(n)}&\leq&\limsup \frac1n\log\prob{T_{\lambda \xi(n)}\leq n}.\label{ineq_incl2}
\end{eqnarray}
Let $A_n := \{T_{\lambda \xi(n)} \leq n, \inf_{k\leq T_{\lambda \xi(n)}}X_k\geq 0 \}$. 
We have $\prob{A_n}\leq \prob{T_{\lambda \xi(n)}\leq n}$.
We will now prove that $\prob{T_{\lambda \xi(n)}\leq n}\leq C\prob{A_n}$ for certain constant $C$ independent on $n$.  

\begin{lemma}\label{environment modification}For each $\omega\in \{T_{\lambda \xi(n)}\leq n\}$ there exists an $\omega' \in A_n$ such that 
$$\omega' \se{\lambda \xi(n),2}\omega.$$ 
\end{lemma}
\begin{proof} Assume the contrary and  consider the set
$$\mathcal G(\omega)= \left\{\omega'\in \{T_{\lambda \xi(n)}\leq n\}: \omega' \se{\lambda \xi(n),2} \omega\right\}$$ and an element $\sigma\in\mathcal G(\omega)$ such that 
$$T_{\lambda \xi(n)}(\sigma)=\min_{\omega'\in \mathcal G(\omega)} T_{\lambda \xi(n)}(\omega').$$

Let us define the following times: \begin{eqnarray*}\alpha&=&\sup\{k<T_{(-\infty,0)}(\sigma):X_k(\sigma)\geq 0\},\\ \beta&=&\inf\{k> \alpha: X_k(\sigma)\geq 0\}.\end{eqnarray*} Clearly, $X_{\alpha}, X_{\beta}\in\{0,1\}$ since $L=2$. If 
$L_{\alpha}(\sigma, X_{\alpha})< M-1$, then we can apply Lemma \ref{raising stacks} to $a=X_{\alpha}$, $b=X_{\beta}$, $t_a=\alpha$, and $t_b=\beta$. We obtain an environment $\sigma'\se{\lambda \xi(n),2}\sigma$ in which at least one visit to $(-\infty, 0)$ is avoided implying that $T_{\lambda \xi(n)(\sigma')}<T_{\lambda \xi(n)}(\sigma)$. Therefore $L_{\alpha}(\sigma, X_{\alpha})\geq M-1$ which implies that $X_{\alpha}\neq X_{\beta}$ and $\{X_{\alpha},X_{\beta}\}=\{0,1\}$. Let $\sigma'$ be the environment obtained from $\sigma$ in the following way:
\begin{eqnarray*}
\sigma'(j,x)= \begin{cases}
\sigma(j,x) & \mbox{for }(j,x)\neq (M-1,X_{\alpha})\\
X_{\beta}-X_{\alpha}& \mbox{for } (j,x)=(M-1,X_{\alpha}).\end{cases}
\end{eqnarray*}
Since every visit to $(-\infty,0)$ in $\sigma$ must start from $X_{\alpha}$ and end at $X_{\beta}$ we conclude that $\sigma'\se{\lambda \xi(n),2}\sigma$ and  $X_{0, T_{\lambda \xi(n)}(\sigma')}\cap (-\infty,0)=\emptyset$ which contradicts our minimality assumption on $\sigma$. This completes the proof of Lemma \ref{environment modification}.
\end{proof}
In the same way as in the proof of inequality (\ref{main inequality}) we now establish $$\prob{T_{\lambda \xi(n)}\leq n}\leq C\prob{A_n}.$$ An argument analogous to the one presented in the proof of Theorem \ref{large_deviations} allows us to prove the existence of the function $\varphi$ such that 
 $$\lim_{n\to\infty}\frac1n\log\prob{A_n}=\varphi(\lambda).$$  
The inequalities (\ref{ineq_incl1}) and (\ref{ineq_incl2}) allow us to conclude (\ref{limit_for_clt}).
\end{proof}

\section{Properties of the rate function}\label{concave}
The next theorem states that the rate function $\phi$ from (\ref{main_limit}) is concave in $\lambda$. 
\begin{theorem}\label{concavity of the rate} Assume that $\alpha,\beta>0$ are real numbers such that $\alpha+\beta=1$. Then for any $\lambda$, $\mu>0$ the following inequality holds $$\phi(\alpha\lambda+\beta\mu)\geq \alpha \phi(\lambda)+\beta\phi(\mu).$$
Moreover, $\phi(0)=0$, $\phi(\lambda)<0$ for $\lambda> 0$, and $\phi(\lambda)=-\infty$ for $\lambda>L$.
\end{theorem}
\begin{proof} Assume that $(\alpha_n)_{n=1}^{\infty}$ and $(\beta_n)_{n=1}^{\infty}$ are sequences of rational numbers for which $n\alpha_n$, $n\beta_n\in\mathbb N$, $0\leq \alpha-\alpha_n\leq\min\left\{\frac1{2\lambda n},\frac1n\right\}$, and $0\leq\beta-\beta_n\leq \min\left\{\frac1{2\mu n},\frac1n\right\}$. 
Notice that \begin{eqnarray*}&&\left\{T_{(\alpha\lambda+\beta\mu)n}\leq n, \inf_{0\leq k\leq T_{(\alpha\lambda+\beta\mu)n}}X_k\geq 0\right\}\\&\supseteq&
\left\{T_{(\alpha\lambda+\beta\mu)n}\leq n, \inf_{0\leq k\leq T_{(\alpha\lambda+\beta\mu)n}}X_k\geq 0, T_{\alpha_n\lambda n}\leq \alpha_n n \right\}.
\end{eqnarray*}
Analogously as in the proof of Theorem \ref{large_deviations} we define the process $\hat X$ as $$\hat X_k(\omega)=X_{k+T_{\alpha_n\lambda n}}(\omega)-X_{T_{\alpha_n\lambda n}}(\omega).$$ Also, denote by $\hat T_x$ the hitting time of the walk $\hat X$, i.e. $$\hat T_x=T_{X_{T_{\alpha_n\lambda n}}+x}-T_{\alpha_n\lambda n}.$$ We now obtain 
\begin{eqnarray} \nonumber&&
\left\{T_{(\alpha\lambda+\beta\mu)n}\leq n, \inf_{0\leq k\leq T_{(\alpha\lambda+\beta\mu)n}}X_k\geq 0, T_{\alpha_n\lambda n}\leq \alpha_n n \right\}\\
 \label{incl}&\supseteq& 
\left\{T_{(\alpha\lambda+\beta\mu)n}\leq n, \inf_{0\leq k\leq T_{(\alpha\lambda+\beta\mu)n}}X_k\geq 0, T_{\alpha_n\lambda n}\leq \alpha_n n, \right.\\ \nonumber &&\left. 
\inf_{0\leq k\leq T_{\alpha_n\lambda n}}X_k\geq 0, 
\hat T_{\beta_n\mu n}\leq \beta_nn,
\inf_{0\leq k\leq \hat T_{\beta_n\mu n}}\hat X_k\geq 0
\right\}.
\end{eqnarray}
Let us define by $\tilde X$ the walk defined as $\tilde X_k=\hat X_{k+\hat T_{\beta_n\mu n}}-\hat X_{\hat T_{\beta_n\mu n}}$, and by $\tilde T_x$ the stopping time $$\tilde T_x=\hat T_{\hat X_{\hat T_{\beta_n\mu n}}+x}-\hat T_{\beta_n\lambda n}.$$ Now we can conclude from the inclusion (\ref{incl}) that
\begin{eqnarray*} \nonumber&&
\left\{T_{(\alpha\lambda+\beta\mu)n}\leq n, \inf_{0\leq k\leq T_{(\alpha\lambda+\beta\mu)n}}X_k\geq 0, T_{\alpha_n\lambda n}\leq \alpha_n n \right\}\\
&\supseteq&
\left\{T_{(\alpha\lambda+\beta\mu)n}\leq n, \inf_{0\leq k\leq T_{(\alpha\lambda+\beta\mu)n}}X_k\geq 0, T_{\alpha_n\lambda n}\leq \alpha_n n, \right.\\ \nonumber &&\left. 
\inf_{0\leq k\leq T_{\alpha_n\lambda n}}X_k\geq 0, 
\hat T_{\beta_n\mu n}\leq \beta_nn,
\inf_{0\leq k\leq \hat T_{\beta_n\mu n}}\hat X_k\geq 0,\right.\\ &&\left. 
\tilde T_{(\alpha\lambda+\beta\mu-\alpha_n\lambda-\beta_n\mu)n}\leq (1-\alpha_n-\beta_n)n,\inf_{0\leq k\leq \tilde T_{(\alpha\lambda+\beta\mu-\alpha_n\lambda-\beta_n\mu)n}} \tilde X_k\geq 0
\right\} \\
&=&
\left\{T_{\alpha_n\lambda n}\leq \alpha_n n, 
\inf_{0\leq k\leq T_{\alpha_n\lambda n}}X_k\geq 0, 
\hat T_{\beta_n\mu n}\leq \beta_nn,
\inf_{0\leq k\leq \hat T_{\beta_n\mu n}}\hat X_k\geq 0,\right.\\ &&\left. 
\tilde T_{(\alpha\lambda+\beta\mu-\alpha_n\lambda-\beta_n\mu)n}\leq (1-\alpha_n-\beta_n)n,\inf_{0\leq k\leq \tilde T_{(\alpha\lambda+\beta\mu-\alpha_n\lambda-\beta_n\mu)n}} \tilde X_k\geq 0
\right\}.
\end{eqnarray*}
From our choice of sequences $(\alpha_n)_{n=1}^{\infty}$ and $(\beta_n)_{n=1}^{\infty}$ we derive the following two inequalities \begin{eqnarray*}\left(\alpha\lambda+\beta\mu-\alpha_n\lambda-\beta_n\mu\right)n&\leq &2,\;\mbox{ and }\\(1-\alpha_n-\beta_n)n&\leq& 2.\end{eqnarray*} The previous inequalities imply \begin{eqnarray*}\prob{\tilde T_{(\alpha\lambda+\beta\mu-\alpha_n\lambda-\beta_n\mu)n}\leq (1-\alpha_n-\beta_n)n,\inf_{0\leq k\leq \tilde T_{(\alpha\lambda+\beta\mu-\alpha_n\lambda-\beta_n\mu)n}} \tilde X_k\geq 0}&\geq& \mu_{\min}^2.\end{eqnarray*} Since the walks $X, \hat X$ and $\tilde X$ occupy disjoint parts of the environment (on the events that there are no backtrackings to the left of $0$), by independence we obtain \begin{eqnarray*}&&
\prob{ T_{(\alpha\lambda+\beta\mu)n}\leq n, \inf_{0\leq k\leq T_{(\alpha\lambda+\beta\mu)n}}X_k\geq 0}
\\&\geq& \prob{T_{\alpha_n\lambda n}\leq \alpha_n n,
\inf_{0\leq k\leq T_{\alpha_n\lambda n}}X_k\geq 0}\\ &&\times \ 
\prob{ \hat T_{\beta_n\mu n}\leq \beta_nn,
\inf_{0\leq k\leq \hat T_{\beta_n\mu n}}\hat X_k\geq 0
} \mu_{\min}^2.
\end{eqnarray*}
Taking logarithms of both sides of the last inequality, dividing by $n$, and taking the limit as $n\to \infty$ we conclude 
\begin{eqnarray}\label{ineq_concavity}
\phi(\alpha\lambda+\beta\mu)&\geq&\lim_{n\to\infty}\frac1n\log\left(\mu_{\min}^2\right)\\
\nonumber
&&+\lim_{n\to\infty}\frac1n \prob{T_{\alpha_n\lambda n}\leq \alpha_n n,
\inf_{0\leq k\leq T_{\alpha_n\lambda n}}X_k\geq 0}\\
\nonumber
&&+ \lim_{n\to\infty}\frac1n\log\prob{ \hat T_{\beta_n\mu n}\leq \beta_nn,
\inf_{0\leq k\leq \hat T_{\beta_n\mu n}}\hat X_k\geq 0
}.
\end{eqnarray}
The first limit on the right-hand side of the last inequality is equal to $0$. For the second limit we use that $\alpha_n n$ is a positive integer, hence 
\begin{eqnarray*}
&&\lim_{n\to\infty}\frac1n \prob{T_{\alpha_n\lambda n}\leq \alpha_n n,
\inf_{0\leq k\leq T_{\alpha_n\lambda n}}X_k\geq 0}\\&=&
\lim_{n\to\infty}\frac{\alpha_n}{\alpha_nn} \prob{T_{\alpha_n\lambda n}\leq \alpha_n n,
\inf_{0\leq k\leq T_{\alpha_n\lambda n}}X_k\geq 0}\\&=&\lim_{n\to\infty}\alpha_n\cdot \lim_{n\to\infty}\frac1{\alpha_nn} \prob{T_{\alpha_n\lambda n}\leq \alpha_n n,
\inf_{0\leq k\leq T_{\alpha_n\lambda n}}X_k\geq 0}\\&=&\alpha \phi(\lambda).
\end{eqnarray*}
Similarly we obtain that the last term on the right-hand side of (\ref{ineq_concavity}) is equal to $\beta\phi(\mu)$ which completes the proof of the convexity.

The equality $\phi(\lambda)=-\infty$ for $\lambda>L$ is trivial because  $\prob{X_n> Ln}=0$. We will now prove the equality $\phi(0)=0$. The event $\{X_n\geq 0\}$ contains the event $\{\omega(i,0)=0$ for all $i\in\{0,1,2,\dots, M-1\}\}$, that is, the event that all cookies at $0$ point to $0$. The probability of this event is greater than $\mu_{\min}^M$, hence $\prob{X_n\geq 0}\geq \mu_{\min}^M$. On the other hand, the complement of $\{X_n\geq 0\}$ contains the event that the first cookie at $0$ points to $-1$ and all cookies at $-1$ are equal to $0$. Thus, $\prob{X_n\geq 0}\leq 1-\mu_{\min}^{M+1}$, and we conclude that $$0=\lim_{n\to\infty}\frac1{n}\log \mu_{\min}^{M}\leq \lim_{n\to\infty}\frac1n\log\prob{X_n\geq 0}\leq\lim_{n\to\infty}\frac1{n}\log \left(1-\mu_{\min}^{M+1}\right)=0.$$

Assume now that $\lambda\in(0,L]$. Let $k=\lfloor \lambda n\rfloor$. We will prove that $\phi(\lambda)\in (-\infty, 0)$ using Lemma \ref{lemma_bound_annulus}. Let $A_k= [-(k+1)L,(k+1)L]\setminus [-kL,kL]$. Since $\{X_n\geq \lambda n\}\subseteq \{T_{A_k}<+\infty\}$ we use Lemma \ref{lemma_bound_annulus} to conclude that $\prob{X_n\geq \lambda n}\leq c^k$ for some constant $c$. For sufficiently large $n$ we have that $\lfloor \lambda n\rfloor\geq \frac{n\lambda}2$ hence $\prob{X_n\geq \lambda n}\leq \left(c^{\lambda /2}\right)^n$. This implies that 
$$\phi(\lambda)\leq\lim_{n\to\infty}\frac1n\log  \left(c^{\lambda /2}\right)^n<0.$$
The finiteness of $\phi(\lambda)$ follows from the fact that $\{T_{\lambda n}\leq n\}$ contains the event $$G=\left\{\omega(0,0)=\omega(0,L)=\omega(0,2L)=\cdots=\omega(0,nL)=L\right\},$$ which is the event that the top cookies at each of the sites $0$, $L$, $2L$, $\dots$, $nL$ point to the location that is $L$ units to its right. The probability of the last event is at least $\mu_{\min}^n$ hence $\phi(\lambda)\geq \log\mu_{\min}$. This completes the proof of the theorem.
\end{proof}

	\bibliographystyle{abbrv}
	\bibliography{cookies}
\end{document}